\documentclass[abstracton]{scrartcl}
\usepackage{amsthm,amssymb,amsmath}
\usepackage{enumitem}
\usepackage{hyperref}
\usepackage{subcaption}
\usepackage{xspace}
\usepackage[utf8]{inputenc}

\usepackage{pgf,tikz,ifthen,calc}
\usetikzlibrary{math,positioning,calc,arrows.meta,decorations.pathmorphing,decorations.markings,shapes,hobby}
\newtheorem{theorem}{Theorem}[section]
\newtheorem{lemma}[theorem]{Lemma}

\newtheorem{observation}[theorem]{Observation}

\theoremstyle{definition}

\theoremstyle{remark}
\newtheorem{clm}{Claim}
\newtheorem*{clmnonum}{Claim}
\newtheorem{remark}[theorem]{Remark}

\newcommand{\MES}{\ensuremath{\mathbf M}}
\newcommand{\UPG}{\ensuremath{\mathbf G}}

\title{Universal planar graphs for the topological minor relation}
\author{Florian Lehner\thanks{Much of the research leading to the results presented in this paper was carried out while the author was supported by the Austrian Science Fund (FWF) Grant no.\ P31889-N35}}

\begin{document}

\maketitle
\begin{abstract}
    Huynh et al.\ recently showed that a countable graph $G$ which contains every countable planar graph as a subgraph must contain arbitrarily large finite complete graphs as topological minors, and an infinite complete graph as a minor.
    We strengthen this result by showing that the same conclusion holds, if $G$ contains every countable planar graph as a topological minor.
    In particular, there is no countable planar graph containing every countable planar graph as a topological minor, answering a question by Diestel and Kühn.
    
    Moreover, we construct a locally finite planar graph which contains every locally finite planar graph as a topological minor. This shows that in the above result it is not enough to require that $G$ contains every locally finite planar graph as a topological minor.
\end{abstract}
\section{Introduction}
Call a graph $U$ \emph{universal} for a graph class $\mathcal G$, if it contains every element of $\mathcal G$, and let us say that $\mathcal G$ \emph{contains a universal element} if there is a universal graph $U$ for $\mathcal G$ which is contained in $\mathcal G$. Depending on the precise definition of containment, this leads to different notions of universality and different universal graphs. 

For instance, a classic result by Pach \cite{zbMATH03735842} states that the class of planar does not contain a universal element with respect to the subgraph relation, thereby providing a negative answer to a question of Ulam. 
In contrast to this, Diestel and Kühn \cite{zbMATH01356639} show that there is a countable planar graph containing all countable planar graphs as minors. This immediately leads to the question for which notions of containment in between the subgraph relation and the minor relation the class of planar graphs contains a universal element. In particular, Diestel and Kühn ask \cite[Problem 6]{zbMATH01356639} whether the class of planar graphs contains a universal element with respect to the topological relation.  Our main result provides a negative answer to this question.

\begin{theorem}
\label{thm:main}
The class of countable planar graphs does not contain a universal graph with respect to the topological minor relation.
\end{theorem}

We point out that Theorem \ref{thm:main} has been proved independently by Krill in his master's thesis \cite{krillmaster}, see also the forthcoming paper \cite{krill}. While our proof is longer and more involved than the one presented in \cite{krillmaster}, it in fact also generalises a significantly stronger result.

In a recent preprint,  Huynh et al.~\cite{huynh2021universality} investigate how sparse a graph which contains all planar graphs as subgraphs can be. In other words, rather than relaxing the notion of containment, they ask how much the requirement of planarity of the universal graph needs to be relaxed in order to get a different answer to Ulam's question. They obtain two complementary results. On the one hand, they show that there are universal graphs for the class of planar graphs which share some key properties with planar graphs. On the other hand, they prove that a universal graph for the class of countable planar graphs with respect to the subgraph relation is in some sense very far from being planar: it contains arbitrarily large complete graphs as topological minors, and the infinite complete graph as a minor. In Section \ref{sec:nounictbl}, we prove the following strengthening of the latter result which immediately implies Theorem \ref{thm:main}.
\begin{theorem}
\label{thm:no-uni-ctbl}
Let $G$ be a countable graph containing every countable planar graph as a topological minor.
\begin{enumerate}
    \item $G$ contains an infinite complete minor.
    \item $G$ contains arbitrarily large finite complete topological minors.
\end{enumerate}
\end{theorem}

In Section \ref{sec:unilocfin}, we turn our attention to locally finite graphs, that is, graphs in which every vertex has only finitely many neighbours.
The main result of this section shows that the conclusion of Theorem \ref{thm:no-uni-ctbl} no longer holds if we only require $G$ to contain all locally finite planar graphs as topological minors.

\begin{theorem}
\label{thm:uni-locfin}
The class of locally finite planar graphs contains a universal element with respect to the topological minor relation.
\end{theorem}

To fully appreciate this result, it is worth mentioning that the aforementioned universality results concerning the subgraph relation and the minor relation are unaffected by restricting to locally finite graphs. Pach's proof from \cite{zbMATH03735842} in fact shows that there is no planar graph containing every locally finite planar graph as a subgraph, and the conclusion of the result by Huynh et al.\ from \cite{huynh2021universality} mentioned above still holds for graphs containing all locally finite planar graphs as subgraphs. Moreover, the universal planar graph for the minor relation constructed in \cite{zbMATH01356639} is in fact locally finite. In light of this, it is perhaps surprising that local finiteness makes a big difference when considering the topological minor relation.

\section{Preliminaries}
%\subsection{Graph basics}

The purpose of this section is to recall basic definitions and set up some notation.
For graph theoretic notions not explicitly defined, we follow \cite{MR3644391}.  

A graph $G$ consists of a set $V(G)$ of vertices and a set $E(G)$ of edges. Given two graphs $G$ and $H$, we denote by $G \cup H$ the graph with vertex set $V(G) \cup V(H)$ and edge set $E(G) \cup E(H)$. Similarly define $G \cap H$. We stress that the graphs in a union do not have to be disjoint, in particular we will often consider unions of graphs which have vertices in common.

It will be convenient to consider cycles in graphs as cyclic sequences of vertices. Given a sequence $X = (x_1,\dots, x_n)$ , a sequence of the form $(x_k, \dots, x_n,x_1,\dots x_{k-1})$ is called a \emph{cyclic shift} of $X$. Call two sequences \emph{cyclically equivalent} if they are cyclic shifts of one another. Call an equivalence class with respect to this relation a \emph{cyclic sequence} and denote the cyclic sequence containing $(x_1,\dots, x_n)$ by $[x_1,\dots, x_n]$. A cycle in a graph $G$ can now be seen as a cyclic sequence $C = [v_1,\dots,v_n]$ of vertices such that $v_iv_{i+1} \in E(G)$ for every $i < n$ and $v_1v_n \in E(G)$. Note that this assigns a direction to the cycle.

Let us call a sequence $Y=(y_1,\dots,y_k)$ \emph{cyclically ordered} with respect to a cyclic sequence $X=[x_1,\dots, x_n]$, if some representative contains $X$ contains $Y$ as a subsequence. If $X$ is clear from the context, we simply call $Y$ cyclically ordered. Clearly, if $Y$ is cyclically ordered, then so is any sequence which is cyclically equivalent to $Y$; we may thus extend this notion to the case where $Y$ is a cyclic sequence.

All (cyclic) sequences considered from now on will be sequences of vertices of graphs. For such a (cyclic) sequence $X$, let us denote the set of all vertices appearing in $X$ by $V(X)$. Let $X_1$ and $X_2$ be cyclic sequences of vertices, let $Y \subseteq V(X_1)$, and let $(y_1, \dots, y_n)$ be a cyclically ordered enumeration of $Y$. We say that a function $\phi\colon Y \to V(X_2)$ \emph{preserves} the cyclic order if $(\phi(y_1), \dots, \phi(y_k))$ is cyclically ordered with respect to $X_2$. We say that such a function \emph{reverses} the cyclic order if $(\phi(y_k), \dots, \phi(y_1))$ is cyclically ordered with respect to $X_2$.

%\subsection{Embeddings and planar graphs}

A plane embedding of a graph $G$ assigns to each vertex $v\in V(G)$ a point $\iota(v) \in \mathbb R^2$ and to each edge $e = uv \in E(G)$ a polygonal arc $\iota(e)$ in $\mathbb R^2$ connecting $\iota(u)$ to $\iota(v)$ such that for any two distinct edges $e$ and $f$ the arcs $\iota(e)$ and $\iota (f)$ are internally disjoint. By a slight abuse of notation we write $\iota(G)$ for $\bigcup_{v \in V(G)} \{\iota(v)\} \cup \bigcup_{e \in E(G)} \iota(e)$. Let us call a graph \emph{planar} if it has a plane embedding. Note that we do not forbid that $\iota(G)$ has accumulation points.

The following theorem by Dirac and Schuster \cite{zbMATH03087498} gives a necessary and sufficient condition for the planarity of countable graphs.

\begin{theorem}
\label{thm:countableplanar}
A countable graph is planar if and only if all of its finite subgraphs are planar.
\end{theorem}

A similar condition for arbitrary graphs has been given by Wagner \cite{zbMATH03246989}.

\begin{theorem}
\label{thm:uncountableplanar}
An arbitrary graph is planar if and only if it has at most continuum many vertices and at most countably many vertices of degree greater than $2$, and all of its finite subgraphs are planar.
\end{theorem}

Similarly to embeddings of graphs in the plane, we can also define embeddings of graphs in other graphs. For two graphs $G$ and $H$, an \emph{$G$-embedding} $\iota$ of $H$ assigns to every $v \in V(H)$ a vertex $\iota (v) \in V(G)$ and to every edge $e = uv \in E(H)$ a $\iota(u)$--$\iota(v)$-path $\iota(e)$ in $G$ such that for distinct edges $e$ and $f$ the paths $\iota(e)$ and $\iota(f)$  are internally disjoint. We call $H$ a \emph{topological minor} of $G$ if there is a $G$-embedding of $H$.

Let us say that two $G$-embeddings $\iota, \iota'$ of a graph $H$ \emph{agree} on $S \subseteq V(H) \cup E(H)$, if for any $s \in S$ we have $\iota(s) = \iota'(s)$. By slight abuse of notation extend this notion to $G$-embeddings of different graphs as follows. Let $S \subseteq V(H) \cup E(H)$ and let $S'\subseteq V(H') \cup E(H')$. Let $f \colon S \to S'$ be a bijection. Let $\iota$ be a $G$-embedding of $H$, and let $\iota'$ be a $G$-embedding of $H'$. We say that $\iota$ and $\iota'$ agree via $f$ on $S$, if $\iota(s) =\iota'(f(s))$ for all $s \in S$. If $H$ and $H'$ have a subgraph in common and all elements in $S$ are contained in this subgraph, then we omit $f$ and tacitly assume that $f$ is the identity. This in particular includes the case where $H$ and $H'$ are obtained from the same graph by adding some vertices and edges.

We say that a family $\mathcal I$ of $G$-embeddings (possibly of different graphs) agrees on $S$ if any pair $\iota, \iota'$ agrees on $S$. We will only need this notion for sets contained in common subgraphs of all involved graphs, hence we do not need to say anything about the functions $f$ involved. We denote the common image of $s \in S$ under all $\iota \in \mathcal I$ by $\mathcal I(s)$.

It is easy to see that Theorems \ref{thm:countableplanar} and \ref{thm:uncountableplanar} do not extend to this notion of embedding. For instance, if $G$ is the disjoint union of all finite graphs, then a countably infinite graph does not have a $G$-embedding, but all of its finite subgraphs do. 

Let $(G_n)_{n \in \mathbb N}$ be an increasing sequence of graphs, that is, $G_n$ is a subgraph of $G_{n+1}$ for every $n \in \mathbb N$. We define $\lim_{n\to \infty} G_n$ as the graph with vertex set $\bigcup_{n\in \mathbb N} V(G_n)$ and edge set $\bigcup_{n\in \mathbb N} E(G_n)$.

\begin{lemma}
\label{lem:limitembedding}
Let $(H_n)_{n \in \mathbb N}$ be an increasing sequence of graphs and let $H = \lim_{n \to \infty} H_n$. If there are $G$-embeddings $\iota_n$ of $H_n$ such that $\iota_{n+1}$ and $\iota_{n}$ agree on $H_n$ for every $n$, then there is a $G$-embedding $\iota$ of $H$ which agrees with $\iota_n$ on $H_n$ for every $n$.
\end{lemma}

\begin{proof}
For $x \in V(H) \cup E(H)$ pick $n$ large enough that $x \in V(H_n) \cup E(H_n)$ and set $\iota(x) = \iota_n(x)$. The conditions of the lemma ensure that this is unambiguous and defines a $G$-embedding of $H$.
\end{proof}

Theorems \ref{thm:countableplanar} and \ref{thm:uncountableplanar}, and Lemma \ref{lem:limitembedding} tell us that we can obtain (plane or $G$-) embeddings of infinite graphs by constructing embeddings of an increasing sequence of finite subgraphs. In the remainder of this section we recall some well known facts and make some easy observations about finite planar graphs.

If $\iota$ is a plane embedding of a finite connected graph $G$, then $\mathbb R^2 \setminus \iota(G)$ consists of finitely many open disks and one unbounded region which we call the \emph{faces} of the embedding; the unbounded region is called the \emph{outer} face, all other regions are called \emph{interior} faces. If $\iota$ is a plane embedding of an infinite graph $G$, we still call a connected component of $\mathbb R^2 \setminus \iota(G)$ a \emph{face} of the embedding. We note that the complement of an embedding of an infinite planar graph can be much more involved due to accumulation points of the embedding; in particular, faces of embeddings of infinite graphs are not necessarily homeomorphic to disks. 

In the following assume that $F$ is a face which is homeomorphic to an open disk. Call a vertex or edge $x$ \emph{incident} to $F$ if $\iota(x)$ lies in the closure of $F$. By tracing the boundary of $F$ in clockwise direction if $F$ is an interior face, or in anti-clockwise direction if $F$ is the outer face, we obtain a cyclic sequence of vertices incident to this face which we call a \emph{facial sequence}. The reason we treat the outer face differently is that we want to make sure that facial sequences are invariant under making a different face the outer face by applying an appropriate inversion. A facial sequence may contain the same vertex more than once (this happens only for cut vertices). If a facial sequence contains each vertex at most once, then this sequence defines a cycle in the graph which we call a \emph{facial cycle}.

\begin{remark}
\label{rmk:glueplanargraphs}
We can combine two connected planar graphs into a larger one by identifying facial cycles. More precisely, let $G$ and $H$ be two graphs, and let $C$ and $C'$ be facial cycles of the same length in $G$ and $H$, and let $\phi \colon V(C) \to V(C')$ be an order reversing bijection. It is not hard to see that the graph obtained by identifying each vertex $v$ of $C$ with $\phi(v)$ is again planar, and that this graph has a plane embedding in which all facial sequences of $G$ and $H$ except $C$ and $C'$ are again facial sequences. Note that if we choose $\phi$ to be order preserving rather than order reversing, then we also obtain a planar graph, but the facial sequences of one of the two graphs are reversed in the combined graph.
\end{remark}

\section{No universal countable planar graph}
\label{sec:nounictbl}

In the proof of Theorem \ref{thm:no-uni-ctbl}, certain grid-like graphs will play an important role, so we start by defining these graphs and collecting some observations about them. We note that many of these observations can be seen as special cases of more general results, see \cite{huynh2021universality}. We still provide proofs where appropriate in order to make this paper as self-contained as possible.

The \emph{triangular wedge} $W$ is the graph with vertex set $\mathbb N_0 \times \mathbb N_0$ and edges between vertices whose coordinates differ by $(1,0)$, $(0,1)$ or $(1,-1)$, see Figure \ref{fig:wedge}. The \emph{$k$-th layer} $W_k$ of $W$ is the subgraph induced by the vertices whose coordinates $(i,j)$ satisfy $i+j = k$. For $0 \leq i \leq k$, let $w_k^i$ be the vertex with coordinates $(i,k-i)$, in other words $w_k^i$ is the $i$-th vertex of $W_k$ in left-to-right order. For $m < n$, define the \emph{$m$--$n$-strip} $W_{m,n}$ as the subgraph of $W$ induced by the vertices in $\bigcup_{m\leq k \leq n} W_k$. Assume that $W$ is a subgraph of some larger graph. A \emph{bypass} for $W_{m,n}$ is a $w_a^0$--$w_b^b$-path with $m<a,b<n$ which meets $W$ only at its endpoints. A bypass $P$ from $w_a^0$  to $w_b^b$ and bypass $P'$ from $w_{a'}^0$  to $w_{b'}^{b'}$ are said to be \emph{crossing} if either $a < a'$ and $b' < b$, or  $a > a'$ and $b' > b$.

\begin{figure}
    \centering
    \begin{tikzpicture}[scale=.7, use Hobby shortcut,
        vertex/.style={inner sep=1pt,circle,draw,fill}]
      \path[clip] (-.5,-.5) rectangle (10.5,6.5);
      
      \path[fill=lightgray] (0,2)--(2,0)--(6,0)--(0,6)--cycle;
      
      \foreach \i in {0,...,30}
      {
        \draw (0,\i)--(20,\i);
        \draw (\i,0)--(\i,20);
        \draw (0,\i)--(\i,0);
      } 
      \draw[ultra thick] (5,0)--(0,5);
      \end{tikzpicture}
      \caption{Triangular wedge. The bold path is the $5$-th layer $W_5$, the subgraph in the shaded area is the $2$--$6$-strip $W_{2,6}$.}
      \label{fig:wedge}
\end{figure}
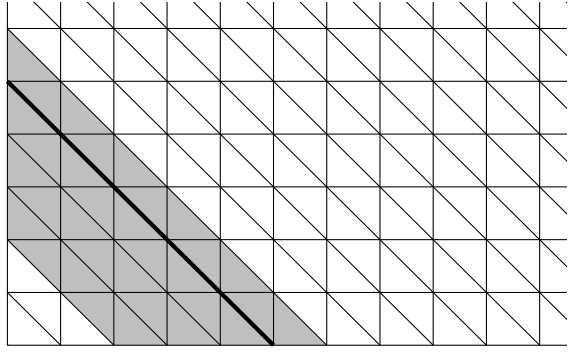

The following observation should be clear, see Figure \ref{fig:permutationpaths} for a sketch.
\begin{observation}
\label{obs:permutationpaths}
Let $m < n$ and let $P$ and $P'$ be disjoint, crossing bypasses for $W_{m,n}$.
The following statements hold for any $k\leq m$.
\begin{enumerate}[label=(\arabic*)]
    \item There are disjoint paths in $W_{m,n}$ connecting $w_m^i$ to $w_n^i$ for $0 \leq i \leq k$.
    \item There are disjoint paths in $W_{m,n} \cup P$ connecting $w_m^i$ to $w_n^{(i+1 \mod k)}$ for $0 \leq i \leq k$.
    \item There are disjoint paths in $W_{m,n} \cup  P \cup P'$ connecting $w_m^0$ to $w_n^k$, $w_m^k$ to $w_n^0$, and $w_m^i$ to $w_n^i$ for $0 < i < k$.
\end{enumerate}
The paths in all three cases can be chosen such that they intersect $W_m$ only in their initial vertices and $W_n$ only in their terminal vertices.
\end{observation}

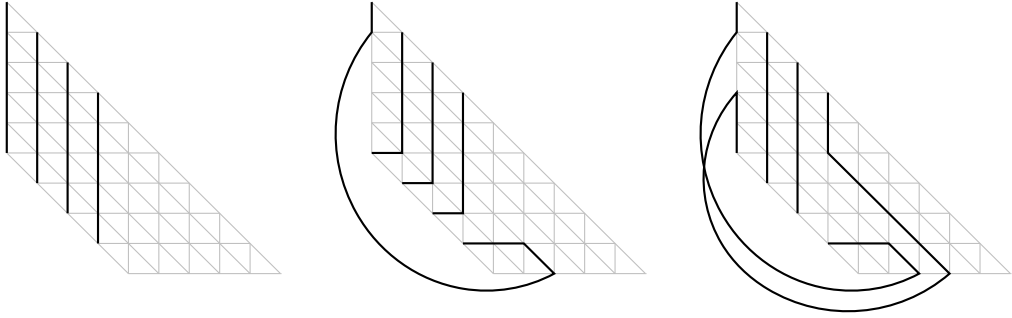
\begin{figure}
    \centering
    \begin{tikzpicture}[scale=.4, use Hobby shortcut,
        vertex/.style={inner sep=1pt,circle,draw,fill}]
      
      \foreach \i in {4,...,9}
        \draw[lightgray] (0,\i)--(\i,0);
      
      \foreach \i in {0,...,9}
      {
        \pgfmathtruncatemacro{\b}{max(0,4-\i)}
        \pgfmathtruncatemacro{\t}{9-\i)}
        \draw[lightgray] (\b,\i)--(\t,\i);
        \draw[lightgray] (\i,\b)--(\i,\t);
      }
      
      \foreach \i in {0,...,3}
        \draw[thick] (\i,4-\i) -- (\i,9-\i);

    \begin{scope}[xshift=12cm]
           \foreach \i in {4,...,9}
        \draw[lightgray] (0,\i)--(\i,0);
      
      \foreach \i in {0,...,9}
      {
        \pgfmathtruncatemacro{\b}{max(0,4-\i)}
        \pgfmathtruncatemacro{\t}{9-\i)}
        \draw[lightgray] (\b,\i)--(\t,\i);
        \draw[lightgray] (\i,\b)--(\i,\t);
      }
      
      \foreach \i in {0,...,2}
        \draw[thick] (\i,4-\i)--++(1,0)--++(0,4);
      \draw[thick] (3,1)--++(2,0)--++(1,-1)..(3,-.5)..(-.5,2)..(0,8)--(0,9);
      \end{scope}
      
       \begin{scope}[xshift=24cm]

           \foreach \i in {4,...,9}
        \draw[lightgray] (0,\i)--(\i,0);
      
      \foreach \i in {0,...,9}
      {
        \pgfmathtruncatemacro{\b}{max(0,4-\i)}
        \pgfmathtruncatemacro{\t}{9-\i)}
        \draw[lightgray] (\b,\i)--(\t,\i);
        \draw[lightgray] (\i,\b)--(\i,\t);
      }
      
      \foreach \i in {1,...,2}
        \draw[thick] (\i,4-\i)--++(0,5);
      \draw[thick] (3,1)--++(2,0)--++(1,-1)..(3,-.5)..(-.5,2)..(0,8)--(0,9);
      \draw[thick] (0,4)--(0,6)..(-.5,1)..(1,-.5)..(7,0)--(3,4)--(3,6);
      \end{scope}
      
      \end{tikzpicture}
      \caption{Routing paths in $W$ with crossing bypasses.}
      \label{fig:permutationpaths}
\end{figure}

The following lemma is a key ingredient in the proof of Theorem \ref{thm:no-uni-ctbl}. It is an immediate consequence of the above observations.
\begin{lemma}
\label{lem:tangled}
Let $(P_i)_{i \in \mathbb N}$ be an infinite family of disjoint bypasses for $W = W_{0,\infty}$ containing infinitely many crossing pairs of bypasses. Denote by $W_{m,n}^+$ the union of $W_{m,n}$ with all $P_i$ that are bypasses for $W_{m,n}$.
\begin{enumerate}[label=(\arabic*)]
    \item \label{itm:permutepaths}
    For any $k < m$, and any permutation $\pi$ of $\{0, \dots, k\}$ there is some $n >m$ and a set of disjoint paths in $W_{m,n}^+$ connecting $w_m^i$ to $w_n^{\pi(i)}$ for $0 \leq i \leq k$.
    \item \label{itm:mappaths}
    For any $k < m$, and any involution $\phi$ of $\{0, \dots, k\}$ there is some $n >m$ and a set of disjoint paths in $W_{m,n}^+$ connecting $w_m^i$ to $w_m^{\phi(i)}$ for $0 \leq i \leq k$.
    \item \label{itm:infinite-adjacent-paths}
    There is an infinite family $\mathcal P$ of (infinite) pairwise disjoint paths in $W^+ = W_{0,\infty}^+$ such that every pair of paths in $\mathcal P$ is connected by an edge.
\end{enumerate}
\end{lemma}

\begin{proof}
Note that by reordering the sequence of bypasses we may without loss of generality assume that $P_i$ and $P_{i+1}$ are crossing for infinitely many $i$.

For the proof of \ref{itm:permutepaths}, note that the cyclic permutation $x \mapsto x+1 \mod k$ and the transposition of $1$ and $k$ generate the symmetric group on $k$ elements. Further note that disjointness of the $P_i$ implies that any vertex of $W$ appears as an endpoint of at most one $P_i$. Hence for any $l$, all but finitely many $P_i$ lie in $W_{l,\infty}^+$, so there must be some $l'$ such that $W_{l,l'}^+$ contains a pair $P_i, P_{i+1}$ of bypasses satisfying the conditions of Observation~\ref{obs:permutationpaths}. Thus we can iterate Observation~\ref{obs:permutationpaths} and concatenate the corresponding paths to obtain the desired family of paths.

For \ref{itm:mappaths}, simply apply \ref{itm:permutepaths} to some permutation $\pi$ such that $|\pi(i) - \pi(j)| = 1$ whenever $\phi(i) = j$. Connecting the $w_m^i$--$w_n^{\pi(i)}$-paths by edges on $W_n$ gives the desired path family.
    
The following recursive construction proves \ref{itm:infinite-adjacent-paths}. Start with the single path consisting of a vertex $v_1^0$. Assume that we have for some $m$ and $k$ constructed disjoint paths in $W_{0,m}^+$ which end at $w_m^i$ for $0 \leq i \leq k$. Applying \ref{itm:permutepaths} with different permutations and concatenating the resulting paths, we obtain $n > m$ and a family of disjoint paths in $W_{0,n}^+$ ending at $w_n^i$ for $0 \leq i \leq k$ such that any pair of them is connected by an edge. Add the path consisting of the single vertex $w_n^{k+1}$ to this family and iterate. In the limit, we get infinitely many disjoint paths any pair of which is connected by an edge.
\end{proof}

Much of the remainder of this section will be devoted to constructing families of bypasses which will enable us to apply the above lemma.

Let $X$, $Y$, and $Z$ be isomorphic copies of $W$. We denote by $X_k$, $Y_k$, $Z_k$, $x_k^i$, $y_k^i$, $z_k^i$, $X_{m,n}$, $Y_{m,n}$, and $Z_{m,n}$ the respective copies of $W_k$, $w_k^i$, and $W_{m,n}$. The \emph{triple wedge} $\overline W$ is the graph obtained from the disjoint union of $X$, $Y$, and $Z$ by adding edges between  $x_k^k$ and $y_k^0$, between $y_k^k$ and $z_k^0$, and between $z_k^k$ and $x_k^0$ for all $k \in \mathbb N_0$, see Figure \ref{fig:triplewedge}. The \emph{annulus} $ \overline W_{m,n}$ is the subgraph of $\overline W$ induced by the vertices of $X_{m,n}$, $Y_{m,n}$, and $Z_{m,n}$. Call a vertex of $\overline W$ \emph{enclosed}, if its neighbourhood is entirely contained in one of the three wedges $X$, $Y$, or $Z$.

\begin{figure}
    \centering
    \begin{tikzpicture}[scale=.7, use Hobby shortcut,
        vertex/.style={inner sep=1pt,circle,draw,fill}]
      \path[clip] (-8.5,-4.5) rectangle (8.5,4.5);
      
      \begin{scope}[rotate=45]
      \path [fill=lightgray]
      ($(45:.3)+(0:4)$)--($(285:.3)+(330:4)$)--
      ($(285:.3)+(240:4)$)--($(165:.3)+(210:4)$)--
      ($(165:.3)+(120:4)$)--($(45:.3)+(90:4)$)--cycle;
       \path [fill=white]
      ($(45:.3)+(0:2)$)--($(285:.3)+(330:2)$)--
      ($(285:.3)+(240:2)$)--($(165:.3)+(210:2)$)--
      ($(165:.3)+(120:2)$)--($(45:.3)+(90:2)$)--cycle;
      
     \foreach \j in {0,120,240}
      {
      \begin{scope}[rotate=\j,shift={(45:.3)}]
      \foreach \i in {0,...,30}
      {
        \draw (0,\i)--(20,\i);
        \draw (\i,0)--(\i,20);
        \draw (0,\i)--(\i,0);
      } 
      \end{scope}
      \foreach \i in {0,...,20}
      {
        \draw[rotate=\j] ($(45:.3)+(0:\i)$)--($(285:.3)+(330:\i)$);
      }
      }
      
      \end{scope}
      \end{tikzpicture}
      \caption{The triple wedge $\overline{W}$. The subgraph in the shaded area is the annulus $\overline W_{2,4}$.}
      \label{fig:triplewedge}
\end{figure}
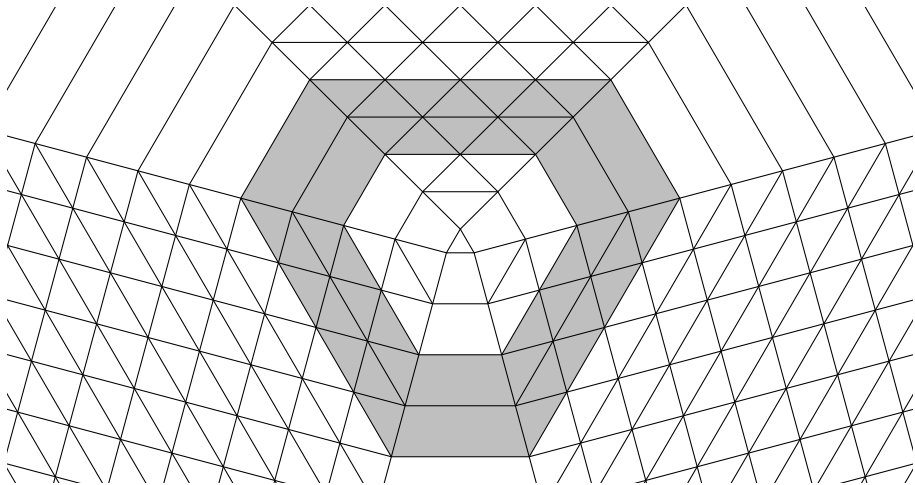

\begin{lemma}
\label{lem:matchingbypass}
Let $u$ and $v$ be non-adjacent vertices in $X_{m+1,n-1} \cup Y_{m+1,n-1}$. Assume that $u$ is an enclosed vertex. The graph obtained from $\overline W_{m,n}$ by adding the edge $uv$ contains a crossing pair of bypasses for $Z$.
\end{lemma}

\begin{proof}
The following notation will be convenient in the proof: If  $a,b$ in $\overline W_k$ are vertices none of whose neighbourhoods is entirely contained in $Z$, then there is a unique $a$--$b$-path in $\overline W_k$ all of whose internal vertices lie in $X$ or $Y$. Let us denote this path by $P(a,b)$.

First assume that $u$ and $v$ lie in different wedges. By symmetry, we may without loss of generality assume that $u \in X$ and $v \in Y$. There are $i,j,k,l$ such that $u = x_k^i$ and $v= y_l^j$. The concatenation of $P(w_k^k,u)$, the edge $uv$ , and $P(v,w_l^0)$ gives a bypass $P$ for $Z$. The concatenation of $P(w_m^m,x_m^m)$, the path $x_m^m,x_{m+1}^{m+1}, x_{m+2}^{m+2},  \dots ,x_n^n$, and $P(x_n^n,w_n^0)$ gives another bypass $P'$ for $Z$.
Note that since $u$ is enclosed, we know that $0<i<k$, hence $P$ does not contain $x_t^t$ for any $t$. Since $m<k,l<n$, we conclude that $P$ and $P'$ are disjoint and crossing.

If $u$ and $v$ lie in the same wedge, we may without loss of generality assume that they both lie in $X$. There are $i,j,k,l$ such that $u = x_k^i$ and $v= x_l^j$. Let us assume that $i \leq j$, the case $i\geq j$ is completely analogous. As before, the concatenation of $P(w_k^k,u)$, the edge $uv$ , and $P(v,w_l^0)$ gives a bypass $P$ for $Z$. Define another bypass $P'$ for $Z$ by concatenation of
$P(w_m^m,x_m^{i-1})$, the path 
\[x_m^{i-1}, x_{m+1}^{i-1}, x_{m+2}^{i-1}, \dots ,x_{k-1}^{i-1},x_{k-1}^{i},x_{k}^{i+1},x_{k+1}^{i+1},x_{k+1}^{i}, x_{k+1}^{i-1}, x_{k+2}^{i-1}, \dots,
x_n^{i-1}\]
and $P(x_n^{i-1}, w_n^0)$. Note that the only vertices $x_s^t$ on $P'$ for which $s>m$ and $t>i$  are neighbours of $u$. If $P$ contained one of these vertices then either $v$ would be a neighbour of $u$, or $i > j$ both of which contradict our assumptions. $m<k,l<n$, we again conclude that $P$ and $P'$ are disjoint and crossing.
\end{proof}

The following construction plays a key role in Lemma \ref{lem:containwplus} below and thus in the proof of Theorem \ref{thm:no-uni-ctbl}. Let $G$ be a graph. 
The \emph{infinitely parallel blow-up} $G^{\parallel}$ of $G$ is the graph obtained from $G$ by replacing every edge  $e$ of $G$ by a countably infinite set of paths of length $2$ connecting the endpoints of $e$. Clearly, all vertices of $G^{\parallel}$ have either infinite degree, or degree $2$. Note that if $G$ is a countable planar graph, then $G^\parallel$ is planar by Theorem \ref{thm:countableplanar}. We call the vertices of infinite degree in (a subdivision of) $G^\parallel$ \emph{original vertices}, and the vertices of degree $2$ \emph{new vertices}. If no confusion is possible, we identify original vertices with the corresponding vertices in $G$, in particular we call a pair of original vertices \emph{adjacent} if the corresponding vertices in $G$ are adjacent.

For $k \in \mathbb N$, let $C_k$ be the $4$-cycle $x_k^k$, $y_k^0$, $y_{k+1}^0$, $x_{k+1}^{k+1}$ in $\overline W$. For $\alpha\in \{0,1\}^\mathbb N$, denote by $\overline W(\alpha)$ the graph obtained from $\overline W$ by adding the edge from $x_k^k$ to $y_{k+1}^{0}$ if $\alpha_k = 0$ and from $x_{k+1}^{k+1}$ to $y_{k}^{0}$ if $\alpha_k = 1$. Note that $\overline W(\alpha)$ is planar since each $C_k$ is a facial cycle in the embedding of $\overline W$ shown in Figure \ref{fig:triplewedge}. Denote by $\overline W_{m,n}(\alpha)$ the subgraph of $\overline W(\alpha)$ induced by the vertices of $\overline W_{m,n}$.

\begin{lemma}
\label{lem:containwplus}
Let $G$ be a countable graph. Let $A \subseteq \{0,1\}^\mathbb N$ be an uncountable set, and assume there are $G$-embeddings $\iota_\alpha$ of $\overline W_{m,\infty}(\alpha)^\parallel$ which agree on $V(\overline W_{m})$. Let $\mathcal I = \{\iota_\alpha \mid \alpha \in A\}$

There is a graph $W^+$ obtained from $W_{m,\infty}$ by adding an infinite family of disjoint bypasses containing infinitely many crossing pairs, and a $G$-embedding $\iota$ of $W^+$ such that there is $\xi \in \{x,y,z\}$ with $\iota(w_m^i) = \mathcal I(\xi_m^i)$ for $0\leq i \leq m$.
\end{lemma}

\begin{proof}To simplify notation, throughout this proof we let $H = \overline W_{m,\infty}$, let $H_n = \overline W_{m,n}$, let $H(\alpha) = \overline W_{m,\infty}(\alpha)$, and let $H_n(\alpha) = \overline W_{m,n}(\alpha)$.

First assume that $\mathcal I$ agrees on $V(H)$. Let $K \subseteq \mathbb N$ be the set of all $k$ for which there are $\alpha, \beta \in A$ with $\alpha_k \neq \beta_k$. Since $\mathcal I$ is infinite, the set $K$ is infinite as well. Let $E^+$ be the union of all $E(H(\alpha))$ for $\alpha \in A$. Note that $E^+$ contains all edges of $\overline W$ and both diagonals of the cycle $C_k$ for every $k \in K$. Let $H^+$ be the graph with vertex set  $V(H)$ and edge set $E^+$. 

It is easy to see that there is $G$-embedding of $H^+$. For every edge $uv \in E^+$, the images of the infinitely many paths connecting $u$ to $v$ in $H(\alpha)^\parallel$ under the embedding $\iota_\alpha$ are internally disjoint from $\mathcal I(V(H))$. Thus we can pick an enumeration of the edges in $E^+$ and for each edge $uv \in E^+$ pick a $\mathcal I(u)$--$\mathcal I(v)$-path $P_{uv}$ which is disjoint from all paths chosen for the preceding edges.

For every $k \in K$, the graph $H^+$ contains a pair of crossing bypasses for $Z$. One is obtained by concatenating the path consisting of all vertices of $X_k$, the edge from $x_k^k$ to $y_{k+1}^0$, and the path consisting of all vertices of $Y_{k+1}$. The other is obtained by concatenating the path consisting of all vertices of $X_{k+1}$, the edge from $x_{k+1}^{k+1}$ to $y_k^0$, and the path consisting of all vertices of $Y_{k}$. Note that among these bypasses there is an infinite disjoint family with infinitely many crossing pairs. Let $W^+$ be the wedge $Z$ together with such a family of bypasses, and let $\iota$ be the restriction of the $G$-embedding of $H^+$ to $W^+$; by definition we have $\iota(w_m^i) = \mathcal I(Z_m^i)$ for $0 \leq i \leq m$. This finishes the proof in case $\mathcal I$ agrees on $V(H)$.

%%%%%%%%%%%%%%%%%%

If $\mathcal I$ contains an uncountable subfamily which agrees on $V(H)$, then the same argument as above can be applied to this subfamily. Hence from now on assume that there is no such subfamily. In the remainder of the proof, we ignore the additional edges in $H(\alpha)$ and view $\mathcal I$ as a family of $G$-embeddings of $H^\parallel$.

For each $i \in \mathbb N_0$ we recursively define 
\begin{itemize}
    \item an integer $n_i$,
    \item a set $M_i$ of edges, and
    \item a $G$-embedding $\iota_i$ of the graph $H_i^+$ obtained from $H_{n_i}$ by adding all edges in $M_i$
\end{itemize}
satisfying the following properties:
\begin{enumerate}[label = (\roman*)]
    \item \label{itm:ngrows}
    $n_0 = m$ and $n_i > n_{i-1}$ for $i > 0$,
    \item \label{itm:mgrows}
    $M_0 = \emptyset$, and $M_i \setminus M_{i-1} = \{m_i\}$ for $i > 0$; the endpoints of $m_i$ are not adjacent in $H$, contained in $V(H_{n_i}) \setminus V(H_{n_{i-1}})$, and at least one of them is enclosed,
    \item \label{itm:iotaagrees}
    $\iota_0$ agrees with $\mathcal I$ on $V(H_m)$, and the restriction of $\iota_i$ to $H_{i-1}^+$ is $\iota_{i-1}$ for $i>0$.
\end{enumerate}

Before carrying out the recursive construction, let us show how the resulting sequences can be used to finish the proof of the lemma. Let $H^+ = \lim_{i \to \infty} H_i^+$. Since $n_i$ and $M_i$ are strictly increasing, $H^+$ is the graph obtained from $H$ by adding the (infinite) set $M = \bigcup_{i \in \mathbb N} M_i$. Note that there is an infinite subset $M' \subseteq M$ such that no edge in $M'$ has an endpoint in one of the three wedges $X$, $Y$, or $Z$; by symmetry we may assume that no edge in $M'$ has an endpoint in $Z$. By Lemma \ref{lem:matchingbypass} the graph $H^+$ contains an infinite set of disjoint bypasses for $Z_{m,\infty}$ containing infinitely many crossing pairs. Let $W^+$ be the union of $Z_{m,\infty}$ and these bypasses. Lemma \ref{lem:limitembedding} tells us that there is a $G$-embedding $\iota$ of $H^+$ which agrees with each $\iota_i$ on $H_i^+$, and thus agrees with $\mathcal I$ on $V(H_{m})$. The restriction of $\iota$ to $W^+$ is the desired $G$-embedding of $W^+$.

It remains to construct the sequences $n_i$, $M_i$, and $\iota_i$. Alongside these sequences, for every $i$ we also construct 
\begin{itemize}
    \item an uncountable family $\mathcal J_i \subseteq \mathcal I$ 
\end{itemize} 
such that 
\begin{enumerate}[resume,label = (\roman*)]
    \item \label{itm:jagrees}
    $\mathcal J_i \cup \{\iota_i\}$ agrees on $V(H_{n_i})$, and
    \item \label{itm:imagedisjoint}
    $\iota_i(H_i^+) \cap \kappa(V(H))\subseteq \kappa(V(H_{n_i}))$ for every $\kappa \in \mathcal J_i$.
\end{enumerate}
Note that for any uncountable family $\mathcal J \subseteq \mathcal I$ and any finite set $S$ of vertices and edges of $H^\parallel$, there is an uncountable subfamily of $\mathcal J$ which agrees on $S$. This is due to the fact that $G$ is countable and thus there are only countably many possible images of $S$. This fact will be used at several points in the construction.

Let $n_0 = m$, and let $M_0 = \emptyset$. For any pair $a,b$ of adjacent vertices in $H_m$ pick a path $P_{ab}$ of length $2$ connecting them. Let $\mathcal J_0 \subseteq \mathcal I$ be an uncountable family which agrees on every $P_{ab}$. Define a $G$-embedding $\iota_0$ of  $H_0^+ = H_m$ by $\iota(v) = \mathcal I(v)$ for every vertex of $H_0^+$, and $\iota(ab) = \mathcal J_0(P_{ab})$ for every edge of $H_0^+$. Properties \ref{itm:ngrows}--\ref{itm:imagedisjoint} are easily seen to hold for $n_0$, $M_0$, $\iota_0$, and $\mathcal J_0$.

The recursive step from $i$ to $i+1$ rests on the following claim.
\begin{clmnonum}
For some $n>n_i$ there is
\begin{enumerate}[label=(\arabic*)]
        \item \label{itm:subfamily}
        an uncountable subfamily $\mathcal J \subseteq \mathcal J_i$ which agrees on $V(H_n)$,
        \item \label{itm:vertexpair}
        $u,v \in V(H_n)\setminus V(H_{n_i})$ which are not adjacent in $H$ such that $u$ is enclosed, and
        \item \label{itm:matchingpath}
        a $\mathcal J(u)$--$\mathcal J(v)$-path $P$ such that $P \cap \iota_i(H_i^+)$ is empty, and $P \cap \kappa(V(H)) = \mathcal J(\{u,v\})$  for every $\kappa \in \mathcal J$.
\end{enumerate}
\end{clmnonum}
Let us assume first that the claim is true. We let $n_{i+1} = n$ and $E_{i+1} = E_i \cup \{uv\}$ for the $n$, $u$, and $v$ provided by the claim. Let $N$ be the number of vertices contained in $\iota_i(H_i^+) \cup P$. For every pair $a,b$ of adjacent vertices in $H_n$ pick $N+1$ different $a$--$b$-paths of length $2$ in $H^\parallel$, and let $\mathcal J_{i+1} \subseteq \mathcal J$ be an uncountable set which agrees on all of these paths. Note that these paths are pairwise internally disjoint, and (by the pigeonhole principle) for each pair $a,b$, the image of at least one such path $P_{ab}$ is internally disjoint from $\iota_i(H_i^+) \cup P$. Let
\[
    \iota_{i+1}(x) =
    \begin{cases}
        \iota_i(x) & \text{if $x \in V(H_i^+)\cup E(H_i^+)$},\\
        \mathcal J_{i+1}(x) & \text{if $x \in  V(H_{n_{i+1}}) \setminus V(H_{n_i})$},\\
        \mathcal J_{i+1}(P_{ab}) &\text{if $x=ab \in E(H_{n_{i+1}}) \setminus E(H_{n_i})$},\\
        P & \text{if $x=uv$}.
    \end{cases}
\]

We briefly check that $\iota_{i+1}$ is a $G$-embedding of $H_{i+1}^+$. First note that the image of any edge is a path connecting the images of the respective endpoints. For edges in $E(H_i^+)$ this follows from the induction hypothesis, for edges in $E(H_{n_{i+1}}) \setminus E(H_{n_i})$ it follows from the fact that every element of $\mathcal J_{i+1}$ is a $G$-embedding of $H^\parallel$ which agrees with $\iota_i$ on $V(H_i^+)$, and for the edge $uv$ it follows from \ref{itm:matchingpath} in the above claim. The path $\iota_{i+1}(uv)$ is internally disjoint from $\iota_i(H_i^+) =\iota_{i+1}(H_i^+)$ by \ref{itm:matchingpath}. The paths $\mathcal J_{i+1}(P_{ab})$ are internally disjoint from $P$ and $\iota_i(H_i^+)$ by definition, and they are internally disjoint from one another since every element of $\mathcal J_{i+1}$ is a $G$-embedding.

Properties \ref{itm:ngrows}, \ref{itm:mgrows}, and \ref{itm:iotaagrees} follow from the above claim and the resulting definitions of $n_{i+1}$, $M_{i+1}$, and $\iota_{i+1}$. For property \ref{itm:jagrees} note that $\mathcal J_{i+1} \cup \{\iota_{i+1}\}$ agrees on $V(H_{n_{i+1}}) \setminus V(H_{n_i})$ by definition of $\iota_{i+1}$ and on $V(H_{n_i})$ by the induction hypothesis since $\mathcal J_{i+1} \subseteq \mathcal J_i$. For \ref{itm:imagedisjoint}, note that $\kappa(V(H))$ contains no internal vertex of $P$ by \ref{itm:matchingpath} above, and no internal vertex of any image $\mathcal J_{i+1}(P_{ab})$ because $\kappa$ is contained in $\mathcal J_{i+1}$. By the induction hypothesis, $\kappa(V(H))$ also does not contain an internal vertex of $\iota_{i+1}(e)= \iota_{i}(e)$ for any $e \in E(H_i^+)$. Hence $\iota_{i+1}(H_{n_i}) \cap \kappa(V(H)) \subseteq \iota_{i+1}(V(H_{n_i}))$, and \ref{itm:imagedisjoint} follows from \ref{itm:jagrees}.

It only remains to provide a proof for the above claim. 
We first show that it suffices to prove that the conclusion of the claim holds when we replace \ref{itm:matchingpath} by the weaker condition that there is
\begin{enumerate}[label = ($3'$)]
    \item \label{itm:matchingpathb}
    a $\mathcal J(u)$--$\mathcal J(v)$-path $P$ such that $P \cap \iota_i(H_i^+)$ is empty, and $P \cap \kappa(V(H_{n+1})) = \mathcal J(\{u,v\})$  for every $\kappa \in \mathcal J$.
\end{enumerate}
Assume that we have found a family $\mathcal J$, vertices $u,v$ and a $\mathcal J(u)$--$\mathcal J(v)$ path $P$ satisfying \ref{itm:subfamily}, \ref{itm:vertexpair}, and \ref{itm:matchingpathb} for some $n$. 
Let $P_j$ be the subpath of $P$ of length $j$ starting at $\mathcal J(u)$. Let $\mathcal K_j \subseteq \mathcal J$ consist of all $\kappa \in \mathcal J$ for which $\kappa(V(H))$ does not contain any internal vertex of $P_j$. Let $k$ be maximal such that $\mathcal K_k$ is uncountable.

If $P_k = P$, then we can simply replace $\mathcal J$ by $\mathcal K_k$ to satisfy the stronger condition \ref{itm:matchingpath}. 
Otherwise, let $t \neq \mathcal J(u)$ be the other endpoint of $P_k$. Since $V(H)$ is countable, there is a vertex $v' \in V(H)$ and an uncountably infinite family $\mathcal K' \subseteq \mathcal K_k$ such that $\kappa (v') = t$ for every $\kappa \in \mathcal K'$. Let $n'$ be such that $v' \in V(H_{n'})$, and let $\mathcal K'' \subseteq \mathcal J'$ be an infinite subfamily which agrees on $V(H_{n'})$. The image of $V(H)$ under $\kappa \in \mathcal K''$ does not contain any internal vertex of $P_k$ since $\mathcal K'' \subseteq \mathcal J_k$. Further note that $v'\notin V(H_{n+1})$ by condition \ref{itm:matchingpathb}, and thus $u$ and $v'$ are not adjacent. Hence $\mathcal K''$, the pair $u,v'$, and the path $P_k$ satisfy \ref{itm:subfamily}, \ref{itm:vertexpair}, and~\ref{itm:matchingpath}.

Finally, we need to show how to construct a family $\mathcal J$, vertices $u,v$ and a $\mathcal J(u)$--$\mathcal J(v)$ path $P$ satisfying \ref{itm:subfamily}, \ref{itm:vertexpair}, and \ref{itm:matchingpathb} for some $n$.

Let $U_1 = V(H_{n_i})$ and recursively define $U_k$ as the union of $U_{k-1}$ all neighbours (in $H$) of some enclosed vertex $a_k$ which has at most $2$ neighbours outside $U_{k-1}$.
It is not hard to see that we can pick the vertices $a_k$ in this construction such that every vertex of $H$ is contained in some $U_k$. Since $\mathcal J_i$ agrees on $V(H_{n_i})$ but not on $V(H)$ there is some $k$ such that $\mathcal J_i$ agrees on $U_k$ but not on $U_{k+1}$. Let $a = a_k$, let $n$ be large enough that $U_{k+1} \subseteq V(H_n)$, let $\mathcal K \subseteq \mathcal J_i$ be an uncountable family which agrees on $V(H_n)$, and let $\kappa \in \mathcal J_i$ be such that $\mathcal K \cup \{\kappa\}$ does not agree on $U_{k+1}$ (and thus does not agree on the neighbours of $a$).

Suppose first that there is a neighbour $x$ of $a$ and a vertex $y \in V(H_n)$ such that $\kappa(x) = \mathcal K(y)$. If $y$ is not adjacent to $a$, then we can choose an uncountable subfamily $\mathcal J\subseteq \mathcal K$ which agrees on $V(H_{n+1})$,  $u=a$, and $v=y$. Recall that $\kappa$ is an embedding of $H^\parallel$, hence $G$ contains infinitely many internally disjoint $\kappa(s)$--$\kappa(t)$-paths for any pair of adjacent vertices in $H$. In particular, there are infinitely many internally disjoint paths connecting $\mathcal J (a) = \kappa(a)$ to $\kappa (x) = \mathcal J(y)$. Among these paths we find one with the desired properties because $\iota_i(H_i^+)$ and $\mathcal J(V(H_{n+1}))$ are finite, and $\iota_i(H_i^+)$ does not contain $\mathcal J(a)$ or $\mathcal J(y)$ by property \ref{itm:imagedisjoint}.

So we may assume that $y$ is adjacent to $a$; in this case all neighbours of $a$ except $x$ and $y$ are contained in $U_k$. There is another common neighbour of $a$ and $x$ besides $y$, denote this neighbour by $z$. Note that $y$ and $z$ are not adjacent. At least one of $y$ and $z$ is enclosed because $a$ is enclosed, and the non-enclosed neighbours of any enclosed vertex are adjacent. As above, we can choose an uncountable subfamily $\mathcal J\subseteq \mathcal K$ which agrees on $V(H_{n+1})$,  $u=y$ and $v=z$ (or vice versa), and a $\mathcal J (z)$--$\mathcal J (y)$-path with the desired properties among the infinitely many internally disjoint $\mathcal J (z)$--$\mathcal J (y)$-paths in $G$.

The above argument only required $n$ to be large enough for $U_{k+1} \subseteq V(H_n)$; in particular, it also works if we replace $n$ by $n+1$. Hence from now on let us assume that $\mathcal K$ agrees on $V(H_{n+1})$ and that no neighbour $x$ of $a$ satisfies $\kappa(x) = \mathcal K(y)$ for any $y \in H_{n+1}$. The neighbourhood of $a$ contains a path $P$ whose endpoints $u$ and $v$ are not adjacent such that $\mathcal K \cup \{\kappa\}$ agrees on $u$ and $v$ but not on the interior points of $P$. The same argument as above tells us that we can connect the images (under $\kappa$) of any two consecutive vertices of $P$ by a path which is disjoint from $\iota_i(H_i^+)$, and does not intersect $\mathcal K(V(H_{n+1}))$ except possibly in $\mathcal K(u) = \kappa(u)$ or $\mathcal K(v) = \kappa(v)$. The union of these paths contains a $\mathcal K(u)$--$\mathcal K(v)$-path $Q$. The family $\mathcal K$, the pair $u,v$, and the path $Q$ satisfy \ref{itm:subfamily}, \ref{itm:vertexpair}, and \ref{itm:matchingpathb}.
\end{proof}

\begin{proof}[Proof of Theorem \ref{thm:no-uni-ctbl}]
For the first part, note that there must be $G$-embeddings of $\overline W(\alpha)^\parallel$ for every $\alpha$. Uncountably many of these embeddings agree on $V(\overline W_m)$ for any $m$, so by Lemma \ref{lem:containwplus} there is a $G$-embedding of a graph $W^+$ consisting of $W_{m,\infty}$ and an infinite set of disjoint bypasses containing infinitely many crossing pairs. The graph $W^+$ contains an infinite complete minor by Lemma \ref{lem:tangled} \ref{itm:infinite-adjacent-paths}, and thus so does $G$.

For the second part, let $k \in \mathbb N$, let $m=k^2$, and let $H(\alpha)$ be defined as follows. Start with $\overline W_{m,\infty}(\alpha)$, and for each $\xi \in \{x,y,z\}$ add $k$ vertices $\xi_*^0, \dots , \xi_*^{k-1}$ and edges between $\xi_i^*$ and $\xi_{m}^{ik+j}$ for $0 \leq j < k$. It is not hard to see that $H(\alpha)$ is planar, and thus so is $H(\alpha)^\parallel$. Hence there is a $G$-embedding of $H(\alpha)^\parallel$ for every $\alpha$, and an uncountable family $\mathcal I$ of these embeddings agrees on $V(\overline W_{m})$ and every $\xi_*^i$. 

Let $G'$ be the graph obtained from $G$ by removing all images $\mathcal I(\xi_*^i)$ for $\xi \in \{x,y,z\}$ and $0 \leq i < k$. Each embedding $\mathcal I$ gives rise to a $G'$-embedding of $\overline W_{m,\infty}(\alpha)$. 

By Lemma \ref{lem:containwplus} there is a $G'$-embedding $\iota$ of a graph $W^+$ consisting of $W_{m,\infty}$ and an infinite set of disjoint bypasses containing infinitely many crossing pairs such that $\iota(w_{m}^i) = \mathcal I (\xi_{m}^i)$ for $0 \leq i \leq m$ for some $\xi \in \{x,y,z\}$. Lemma \ref{lem:tangled} \ref{itm:mappaths} implies that $W^+$ contains disjoint paths connecting $\xi_{m}^{ki+j}$ to  $\xi_{m}^{kj+i}$ for $0 \leq i < j < k$. The images of these paths under $\iota$ are disjoint $\mathcal I (\xi_{m}^{ki+j})$--$\mathcal I(\xi_{m}^{kj+i})$-paths $P_{ij}$ in $G'$.

There are infinitely many internally disjoint paths in $G$ connecting $\mathcal I (\xi_*^i)$ to $\mathcal I (\xi_{m}^{ki+j})$ since any element of $\mathcal I$ is a $G$-embedding of $H(\alpha)^\parallel$. Among these paths we can inductively find internally disjoint $\mathcal I (\xi_*^i)$--$\mathcal I (\xi_{m}^{ki+j})$-paths $Q_{ij}$ for $0 \leq i, j < k$ which are also internally disjoint from all paths $P_{ij}$.

We can define a $G$-embedding $\kappa$ of the complete graph on $k$ vertices $v_0, \dots v_{k-1}$ by letting $\kappa(v_i) = \mathcal I(\xi_*^i)$, and $\kappa(v_iv_j)$ the union of the paths $Q_{ij}$, $P_{ij}$, and $Q_{ji}$.
\end{proof}

\section{A universal, locally finite, planar graph}
\label{sec:unilocfin}

In this section, we construct a locally finite graph which is universal for the topological minor relation, thus proving Theorem \ref{thm:uni-locfin}.

Let $W \subseteq V(G)$ and let $\phi\colon W \to V(G)$ such that $\phi(W) \cap W =\emptyset$. A \emph{$\phi$-linkage} is a set of disjoint paths containing a $w$--$\phi(w)$-path $P_w$ for every $w \in W$. Let us call two cycles $C_1$ and $C_2$ \emph{well-linked} for any order reversing injection from $W \subseteq V(C_1)$ to $V(C_2)$ there is a $\phi$-linkage whose paths meet $C_1$ and $C_2$ only in their respective endpoints. Note that we take order reversing functions $\phi$ so that two facial cycles in a planar graph can be well linked. A \emph{$m$--$n$-mesh} is a planar graph together with a pair of disjoint  well-linked facial cycles whose lengths are $m$ and $n$ respectively. It is easy to see that $n$-$m$-meshes exist for all $m$ and $n$, for instance we may start with a Cartesian product $C_N \square P_N$ where $N$ is much larger than $n$ and $m$ and connect cycles of length $n$ and $m$ to the two facial cycles of length $N$ in an appropriate way.

For every $n \in \mathbb N$, let $M_1$ and $M_2$ be two $(2n)$-$(n^2)$-meshes. Denote the pair of well-linked cycles in $M_i$ by $C_i = (v_{1,i},\dots, v_{2n,i})$ and $C'_i = (v_{1,i},\dots, v_{n^2,i})$. Let $\MES(n)$ be the graph obtained by adding an additional vertex $z$, connecting $z$ to every vertex of $C_1$, and adding edges between $v'_{kn,1}$ and $v'_{kn,2}$ for $1\leq k \leq n$, see Figure \ref{fig:meshes}. We call $M_1$ the \emph{inner mesh}, $M_2$ the \emph{outer mesh}, and $z$ the \emph{centre} of $\MES(n)$. The cycles $C_1'$ and $C_2'$ are called the \emph{inner perimeter} and \emph{outer perimeter}, respectively, and $C_2$ is called the \emph{boundary} of $\MES(n)$. The edges connecting the outer perimeter to the inner perimeter are called the \emph{spokes} of $\MES(n)$. Moreover, the $n$ cycles of length $2n+2$ consisting of a path of length $n+1$ on $C_1'$ and a path of length $n+1$ on $C_2'$ are called the \emph{attachment cycles} of $\MES(n)$. Note that $\MES(n)$ has a plane embedding such that the boundary and all attachment cycles are facial cycles.

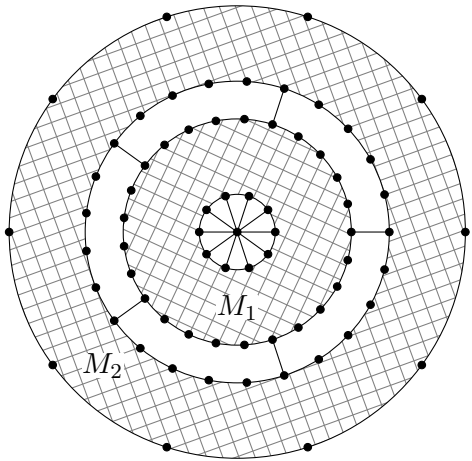
\begin{figure}
    \centering
    \begin{tikzpicture}[vertex/.style={inner sep=1pt,circle,draw,fill}]
      \begin{scope}
      \path[clip] (0,0) circle (3);
      \begin{scope}[scale=.2,rotate=20]
        \foreach \i in {-20,...,20}
        {
          \draw[color=gray] (-20,\i)--(20,\i);
          \draw[color=gray] (\i,-20)--(\i,20);
        }
       \end{scope}
      \path[fill=white] (0,0) circle (2);
      \end{scope}
       
      \begin{scope}
      \path[clip] (0,0) circle (1.5);
      \begin{scope}[scale=.2,rotate=65]
        \foreach \i in {-20,...,20}
        {
          \draw[color=gray] (-20,\i)--(20,\i);
          \draw[color=gray] (\i,-20)--(\i,20);
        }
      \end{scope}
      \path[fill=white] (0,0) circle (.5);
      \end{scope}
        \draw (0,0) circle (.5);
        \draw (0,0) circle (1.5);
        \draw (0,0) circle (2);
        \draw (0,0) circle (3);
        \foreach \i in {1,...,10}
        {
          \draw (0,0)--(36*\i:.5);
          \node[vertex] at (36*\i:.5){};
          \node[vertex] at (36*\i:3){};
        }
        \node[vertex] at (0,0){};
        
        \foreach \i in {1,...,5}
          \draw (360/5*\i:1.5)--(360/5*\i:2);
        \foreach \i in {1,...,25}
        {
          \node[vertex] at (360/25*\i:1.5){};
          \node[vertex] at (360/25*\i:2){};
        }
        \node[inner sep=0,fill=white] at (-90:1){$M_1$};
        \node[inner sep=0,fill=white] at (-135:2.5){$M_2$};
    \end{tikzpicture}
    \caption{Construction of the graph $\MES(5)$. The cycle bounding the outer face in this drawing is the boundary, the cycles bounding the faces between $M_1$ and $M_2$ are the attachment cycles.}
    \label{fig:meshes}
\end{figure}

Using these graphs, we construct a graph $\UPG$ recursively as follows. In each iteration, we have a graph $\UPG(n)$ and a set $\mathcal C_n$ of pairwise disjoint facial cycles with respect to some embedding of $\UPG(n)$ such that each $C \in \mathcal C_n$ has length $2(n+1)$. 

We start the inductive construction by letting $G_1$ be a cycle of length $4$, and choosing $\mathcal C_1$ as the set consisting of this cycle. In each subsequent step, for each cycle $C \in \mathcal C_n$ we take a copy of $\MES(n+1)$ and identify its boundary with $C$ as indicated in Remark~\ref{rmk:glueplanargraphs}. Note that apart from the boundaries, all facial cycles in all copies of $\MES(N+1)$ are facial cycles of $\UPG(n+1)$; in particular, all attachment cycles are facial cycles of length $2(n+2)$. Let the set $\mathcal C_{n+1}$ consist of all attachment cycles of all copies of $\MES(n+1)$. Define $\UPG=\lim_{n \to \infty} \UPG(n)$. This graph is planar by Theorem \ref{thm:countableplanar} since every finite subgraph of $\UPG$ is a subgraph of some $\UPG(n)$ and planarity is preserved under taking subgraphs.

Note that the above construction is not unique (even when $\MES(n)$ is fixed for every $n$) since there are different, non-isomorphic ways of identifying the facial cycles. However, this non-uniqueness will not be an issue; we will show that any choice leads to a universal planar graph which with respect to the topological minor relation.

\begin{theorem}
\label{thm:upg}
Any connected, locally finite, planar graph has a $\UPG$-embedding.
\end{theorem}

\begin{proof}
Let $G$ be a connected, locally finite, planar graph, and let $H$ be the graph obtained from $G$ by subdividing every edge. Clearly, any $\UPG$-embedding of $H$ gives rise to an $\UPG$-embedding of $G$; the path corresponding to an edge $e \in E(G)$ is simply the union of the two paths corresponding to the edges obtained by subdividing $e$. In particular, it suffices to show that $H$ has a $\UPG$-embedding.

We partition the vertices of $H$ into \emph{original vertices}, that is, vertices corresponding to vertices of $H$, and \emph{subdivision vertices}, that is, vertices added to subdivide an edge. Any subdivision vertex has precisely two neighbours both of which are original, and any original vertex only has subdivision neighbours.

Let $(v_n)_{n \in \mathbb N}$ be an enumeration of the original vertices such that the subgraph of $G$ induced by the vertices corresponding to $v_1, \dots, v_n$ is connected for every $n \in \mathbb N$. Let $V_n = \{v_k \mid k \leq n\}$. Let $H_n$ be the subgraph of $H$ induced by $V_n$ and all neighbours of $V_n$. Let $H_n'$ be the be the subgraph of $H$ induced by $V_n$ and all subdivision vertices both of whose neighbours are in $V_n$. Subdivision vertices with exactly one neighbour in $V_n$ are called \emph{loose ends} of $H_n$, they have degree $1$ in $H_n$ and are the only vertices of $H_n$ which are not contained in $H_n'$. Note that all graphs $H_n$ and $H_n'$ are connected by our choice of the enumeration $(v_n)$.

Since $G$ is planar, so is $H$; for the remainder of the proof we fix an arbitrary plane embedding $\iota$ of $H$. Restricting this embedding to $H_n$ or $H_n'$ clearly gives plane embeddings for all $n \in \mathbb N$, by a slight abuse of notation we denote these embeddings by $\iota$ as well. When referring to faces of $H_n$ or faces of $H_n'$ we tacitly assume that these are faces with respect to the embedding $\iota$.  We say that a loose end $v$ of $H_n$ belongs to a face $F$ of $H_n$ if $\iota(v)$ is contained in the closure of $F$. Note that any loose end belongs to exactly one face since it has degree $1$ and thus it has a neighbourhood whose intersection with $\mathbb R^2 \setminus \iota(H_n)$ is connected. Further note that any loose end appears precisely once in the boundary sequence of the face it belongs to. Denote by $L(F)$ the set of loose ends belonging to the face $F$.

We now inductively construct $\UPG(m)$-embeddings of $H_n$ for appropriate choices of $m$ and $n$. Call a $\UPG(m)$-embedding $\phi$ of $H_n$ \emph{good} if there is an injective map assigning to each face $F$ of $H_n$ a cycle $C_F \in \mathcal C_m$ such that the restriction of $\phi$ to $L(F)$ is an order preserving injection from $L(F)$ to $C_F$ (with respect to the cyclic orders given by the boundary sequence of $F$ and $C_F$, respectively).

Our inductive construction rests on the following two claims whose proofs are fairly straightforward. Let $\phi$ be a good $\UPG(m)$-embedding of $H_n$ and assume that $H_n$ has at most $m$ loose ends. Note that applying the two claims inductively gives a sequence of embeddings of $H_n'$ into $\UPG$ satisfying the conditions of Lemma \ref{lem:limitembedding}, thereby finishing the proof of Theorem \ref{thm:uni-locfin}
\begin{clm}
\label{clm:extend1}
There is a good $\UPG(m+1)$-embedding $\psi$ of $H_n$ such that $\phi$ and $\psi$ agree on $H_n'$
\end{clm}
\begin{clm}
\label{clm:extend2}
If $H_{n+1}$ has at most $m$ loose ends, then there is a good $\UPG(m+1)$-embedding $\psi$ of $H_{n+1}$ such that $\phi$ and $\psi$ agree on $H_n'$.
\end{clm}

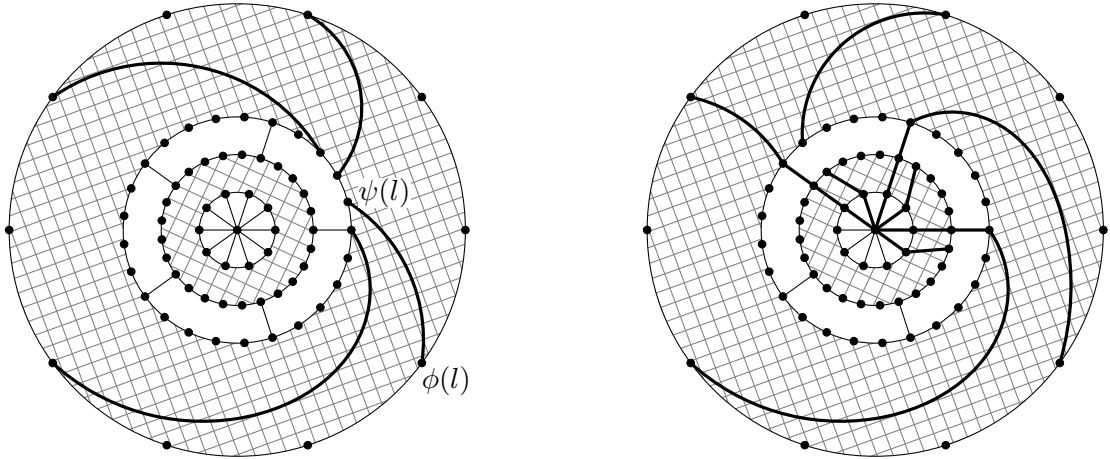
\begin{figure}
    \centering
    \begin{tikzpicture}[use Hobby shortcut,
        vertex/.style={inner sep=1pt,circle,draw,fill}]
      \begin{scope}
      \path[clip] (0,0) circle (3);
      \begin{scope}[scale=.2,rotate=20]
        \foreach \i in {-20,...,20}
        {
          \draw[color=gray] (-20,\i)--(20,\i);
          \draw[color=gray] (\i,-20)--(\i,20);
        }
       \end{scope}
      \path[fill=white] (0,0) circle (1.5);
       \end{scope}
       
      \begin{scope}
      \path[clip] (0,0) circle (1);
      \begin{scope}[scale=.2,rotate=65]
        \foreach \i in {-20,...,20}
        {
          \draw[color=gray] (-20,\i)--(20,\i);
          \draw[color=gray] (\i,-20)--(\i,20);
        }
      \end{scope}
      \path[fill=white] (0,0) circle (.5);
       \end{scope}
        [vertex/.style={inner sep=1pt,circle,draw,fill}]
        \draw (0,0) circle (.5);
        \draw (0,0) circle (1);
        \draw (0,0) circle (1.5);
        \draw (0,0) circle (3);
        \foreach \i in {1,...,10}
        {
          \draw (0,0)--(36*\i:.5);
          \node[vertex] at (36*\i:.5){};
          \node[vertex] at (36*\i:3){};
        }
        \node[vertex] at (0,0){};
        
        \foreach \i in {1,...,5}
          \draw (360/5*\i:1)--(360/5*\i:1.5);
        \foreach \i in {1,...,25}
        {
          \node[vertex] at (360/25*\i:1){};
          \node[vertex] at (360/25*\i:1.5){};
        }
        \draw[very thick] (36*4:3)..(90:2)..(360/25*3:1.5);
        \draw[very thick] (36*2:3)..(50:2.5)..(360/25*2:1.5);
        \draw[very thick] (-36:3)..(-10:2.2)..(360/25*1:1.5);
        \draw[very thick] (-36*4:3)..(-90:2.5)..(-30:2)..(0:1.5);
        
        \node at (-36:3.4){$\phi(l)$};
        \node[inner sep=0,fill=white] at (360/25:2){$\psi(l)$};
    \end{tikzpicture}%
    \hspace{2cm}
    \begin{tikzpicture}[use Hobby shortcut,
        vertex/.style={inner sep=1pt,circle,draw,fill}]
      \begin{scope}
      \path[clip] (0,0) circle (3);
      \begin{scope}[scale=.2,rotate=20]
        \foreach \i in {-20,...,20}
        {
          \draw[color=gray] (-20,\i)--(20,\i);
          \draw[color=gray] (\i,-20)--(\i,20);
        }
       \end{scope}
      \path[fill=white] (0,0) circle (1.5);
       \end{scope}
       
      \begin{scope}
      \path[clip] (0,0) circle (1);
      \begin{scope}[scale=.2,rotate=65]
        \foreach \i in {-20,...,20}
        {
          \draw[color=gray] (-20,\i)--(20,\i);
          \draw[color=gray] (\i,-20)--(\i,20);
        }
      \end{scope}
      \path[fill=white] (0,0) circle (.5);
       \end{scope}
        [vertex/.style={inner sep=1pt,circle,draw,fill}]
        \draw (0,0) circle (.5);
        \draw (0,0) circle (1);
        \draw (0,0) circle (1.5);
        \draw (0,0) circle (3);
        \foreach \i in {1,...,10}
        {
          \draw (0,0)--(36*\i:.5);
          \node[vertex] at (36*\i:.5){};
          \node[vertex] at (36*\i:3){};
        }
        \node[vertex] at (0,0){};
        
        \foreach \i in {1,...,5}
          \draw (360/5*\i:1)--(360/5*\i:1.5);
        \foreach \i in {1,...,25}
        {
          \node[vertex] at (360/25*\i:1){};
          \node[vertex] at (360/25*\i:1.5){};
        }
        \draw[very thick] (36*4:3)..(140:2)..(360/25*10:1.5)--(0,0);
        \draw[very thick] (36*2:3)..(100:2.5)..(360/25*9:1.5);
        \draw[very thick] (-36:3)..(0:2.5)..(50:2)..(360/25*5:1.5)--(0,0);
        \draw[very thick] (-36*4:3)..(-90:2.5)..(-30:2)..(0:1.5)--(0,0);
        
        \draw[very thick] (360/25*9:1)--(360/10*3:.5)--(0,0);
        \draw[very thick] (360/25*4:1)--(360/10*1:.5)--(0,0);
        \draw[very thick] (-360/25:1)--(-360/10:.5)--(0,0);
    \end{tikzpicture}
    \caption{Sketch of the constructions used to prove Claims \ref{clm:extend1} and \ref{clm:extend2}.}
    \label{fig:extendlooseends}
\end{figure}

It remains to prove the two claims. In both claims, for any vertex or edge $x$ of $H_n'$, we must have $\psi(x) = \phi(x)$. Thus the maps $\phi$ and $\psi$ in the above claims only differ in the embeddings of the loose ends, their incident edges and potentially the additional vertex $v_{n+1}$ and its incident edges. We refer to Figure \ref{fig:extendlooseends} for a sketch of how the embeddings of loose ends and their incident edges are extended into the copies of $\MES(m+1)$ that were added in the construction of $\UPG(m+1)$ from $\UPG(m)$.

For a formal proof, consider the following setup. Let $F$ be an arbitrary face of $H_n$. As before, denote by $L(F)$ the loose ends belonging to $F$, and for $l \in L(F)$ let $e_l$ be the unique edge incident to $l$. Let $M_F$ be the copy of $\MES(m+1)$ whose boundary was identified with $C_F$ in the construction of $\UPG(m+1)$. Consider $\phi(L(F))$ as vertices on the boundary of $M_F$, ignoring the embedding of the rest of $H_n$. Recall that the boundary was identified with $C_F$ using an order reversing bijection, so the restriction of $\phi$ to $L(F)$ is an order reversing injection from $L(F)$ to the boundary of $M_F$.

For the proof of Claim \ref{clm:extend1} we apply the following construction for each face $F$ of $H_n$, see the left half of Figure \ref{fig:extendlooseends}. Pick an arbitrary attachment cycle $C_F'$ of $M_F$. Let $Y$ be the set of vertices on $C_F'$ which are contained in the outer mesh of $M_F$; note that $|Y| = m+2 \geq |L(F)|$. Let $\xi\colon \phi(L(F)) \to Y$ be an order reversing injection. Since $M_F$ is a mesh, we can find a $\xi$-linkage whose paths intersect the boundary of $M_F$ only in their endpoints.

Let $\psi(e_l)$ be the concatenation of $\phi(e_l)$ and  the $\phi(l)$--$\xi(\phi(l))$-path in this linkage, let $\psi(l) = \xi(\phi(l))$, and let $\psi(x) = \phi(x)$ for every vertex or edge $x$ of $H_n'$. It is easy to check that $\psi$ is a $\UPG(m+1)$-embedding of $H_n$. By construction, the images of all loose ends belonging to $F$ lie on $C'_F$. Moreover, the cyclic orders of $\phi(v_i)$ on $C_F$ and $\psi(x_i)$ on $C'_F$ coincide since the composition of two order reversing maps is an order preserving map. For any two faces $F_1$ and $F_2$ we have $C_{F_1}\neq C_{F_2}$ and thus $C_{F_1}'$ and $C_{F_2}'$ are attachment cycles of different copies of $\MES(m+1)$, so the function mapping $F$ to $C'_F$ is an injection.

Thus $\psi$ is a good $\UPG(m+1)$-embedding of $H_n$ which coincides with $\phi$ on $H_n'$. This proves Claim \ref{clm:extend1}.

For the proof of Claim \ref{clm:extend2}, we first note that any face $F$ of $H_{n+1}$ which is not incident to $v_{n+1}$ is also a face of $H_n$. We can thus apply the same construction as above to $F$ to obtain an attachment cycle of $M_F$ in $\UPG(m+1)$ and an appropriate embedding of the loose ends belonging to $F$ and their incident edges.

It remains to provide a construction for faces $F$ incident to $v_{n+1}$.
A sketch of this construction is shown in the right half of Figure \ref{fig:extendlooseends}.
%Observe that any face $F$ of $H_{n+1}$ incident to $v_{n+1}$ is contained in the face $F_0$ of $H_n$ containing the embedding $\iota(v_{n+1})$. It is  an easy observation that the boundary sequence of any such $F$ can be obtained from a subsequence of the boundary sequence of $F_0$ by adding $v_{n+1}$ and some loose ends incident to $v_{n+1}$.

Let $l_1,\dots,l_k$ be a cyclically ordered enumeration of the loose ends $L(F_0)$. Let $i_1, \dots, i_r$ be the indices such that $l_{i_j}$ is incident to $v_{n+1}$, for convenience set $i_{r+1} = i_1$. Let $B(F_0)$ be the boundary sequence of $F_0$ and let $B_j(F_0)$ be the part of $B(F_0)$ strictly between $l_{i_j}$ and $l_{i_{j+1}}$; if $v_{n+1}$ is incident to a unique $l \in L(F_0)$, then $B_1(F_0)$ is the whole boundary sequence without $l$, cyclically permuted so it starts with the successor of $l$. Clearly, $B_j(F_0)$ contains at most $|L(F_0)|-1$ vertices, all of which are loose ends belonging to the same face of $H_{n+1}$. For $j \neq j'$, the loose ends in $B_j(F_0)$ and $B_{j'}(F_0)$ belong to different faces of $H_{n+1}$.

Let $Y$ be the set of vertices in the outer perimeter of $M_{F_0}$ in clockwise cyclic order (that is, we consider the outer perimeter as a face of the outer mesh). Pick an order reversing injection $\xi \colon \phi(L(F_0)) \to Y$ such that every $\phi(l_{i_j})$ is incident to a spoke, and the image of each $B_j(F_0)$ is completely contained in an attachment cycle $C_j$. This is possible, because $M_{F_0}$ has at least $m \geq |L(F_0)|$ spokes, and between any two spokes we can find $m-1 \geq  |L(F_0)|-1$ vertices belonging to the same attachment cycle.

Next, let $N$ be the set of neighbours of $v_{n+1}$. Going around $\iota(v_{n+1})$ in clockwise direction defines a cyclic order on $N$. The restriction of this cyclic order to the vertices $l_{i_j}$ agrees with the restriction of the boundary sequence of $F_0$ to these vertices, otherwise the embedding would not be planar. Let $Z$ be the set of vertices on the inner perimeter of $M_{F_0}$ in clockwise cyclic order. Let $\eta \colon N \to Z$ be an order preserving map such that $\eta(l_{i_j})$ is incident to $\xi(\phi(l_{i_j}))$ and the vertices between $l_{i_j}$ and $l_{i_{j+1}}$ are all mapped to the attachment cycle $C_j$ defined above.

Now define the embedding $\psi$ as follows. Let $\psi(x_{n+1})$ be the centre $z$ of $M_{F_0}$. For $x \in N$ let $\psi(x) = \eta(x)$. Disjoint $z$--$\eta(x)$-paths for the images $\psi(v_{n+1}x)$ can be constructed from a linkage between $\eta(N)$ and the neighbours of $z$. Such a linkage exists in the inner mesh of $M_{F_0}$ because $v_{n+1}$ has $N < |L(F_0)| + |\{\text{loose ends of }H_{n+1}\}| \leq 2m$ neighbours, so there is an injection from $\eta(N)$ to the cycle consisting of the neighbours of $z$. 
Next fix a $\xi$-linkage in the outer mesh of $M_{F_0}$ whose paths intersect the boundary of $M_{F_0}$ only in their endpoints.
For $l \in L(F_0) \setminus N$ we set $\psi(l) = \xi(\phi(l))$ and let $\psi(e_l)$ be the concatenation of $\phi(e_l)$ with the $\phi(l)$--$\xi(\phi(l))$-path in this linkage.
For $l \in N$, let $\psi(e_l)$ be the concatenation of $\phi(e_l)$ with the $\phi(l)$--$\xi(\phi(l))$-path in this linkage and the incident spoke of $M_{F_0}$; note that $\psi(l) = \eta(l)$ is the other endpoint of this spoke.
Finally, let $\psi(x) = \phi(x)$ for every vertex or edge $x$ of $H_n'$. 

This clearly gives a $\UPG(m+1)$-embedding of $H_{n+1}$. By definition, $\phi$ and $\psi$ coincide on $H_n'$. To see that $\psi$ is a good embedding, note that the boundary sequence of each face $F$ of $H_{n+1}$ incident to $v_{n+1}$ has the form
\[l_{i_j},B_j(F_0),l_{i_{j+1}},v_{n+1},l_1',v_{n+1},l_2',v_{n+1}\dots v_{n+1}l_s'\]
for some $s \geq 0$, where $l_1', \dots, l_s'$ is the reversal of (possibly empty) sequence of neighbours of $v_{n+1}$ appearing between $l_{i_j}$ and $l_{i_{j+1}}$ in the cyclic order. The order of the loose ends in this sequence coincides with the cyclic order of their embeddings on the cycle $C_j$.
\end{proof}

Theorem \ref{thm:uni-locfin} is now an easy consequence of the above result and Theorem \ref{thm:uncountableplanar}.

\begin{proof}[Proof of Theorem \ref{thm:uni-locfin}]
Let $G$ be the disjoint union of the following graphs:
\begin{enumerate}
    \item countably many copies of $\UPG$,
    \item continuum many copies of the cycle $C_k$ for every $k \in \mathbb N$,
    \item continuum many double rays.
\end{enumerate}
This graph is planar by Theorem \ref{thm:uncountableplanar}, and it is locally finite since all of the constituent graphs are locally finite.

If $H$ is a locally finite planar graph, then by Theorem \ref{thm:uncountableplanar} there are at most countably many components of $H$ containing a vertex of degree $3$ or more. Each such component can be embedded into a different copy of $\UPG$ in $G$. There are at most continuum many other components all of which are either cycles or (possibly infinite) paths. Each of these components can be embedded into a different copy of some cycle $C_k$ or double ray.
\end{proof}

\begin{remark}
Theorem \ref{thm:upg} shows that the class of connected locally finite planar graphs also contains a universal element with respect to the topological minor relation. The above proof of Theorem \ref{thm:countableplanar} can be easily adapted to yield the same conclusion for the class of countable, locally finite planar graphs.
\end{remark}

\begin{remark}
The construction of $\MES$ can be modified to give a graph with maximum degree $d$ for any $d \geq 3$. The graphs $\MES(n)$ can be built in a way that every vertex except the centre has degree at most $3$ (replace vertices of higher degree by appropriate cycles) and the centre has degree $d$ (do not connect it to all vertices on the cycle of length $2n$). Using this modified construction, it is straightforward to check that the above proof shows that for every $d \in \mathbb N$, the class of planar graphs with maximum degree at most $d$ contains a universal graph with respect to the topological minor relation (the cases $d \leq 2$ are trivial), and the same is true for the class of connected planar graphs with maximum degree at most $d$.
\end{remark}

\bibliographystyle{abbrv}
\bibliography{bibliography.bib}

\end{document}